\documentclass[12pt,reqno]{amsart}
\usepackage[dvips]{epsfig}
\RequirePackage{amsmath}
\usepackage{epsfig}
\usepackage{graphicx,epsfig}
\usepackage{amscd,amsthm,amsfonts,amsopn,amssymb,verbatim}
\usepackage{color}
\usepackage{times}
\numberwithin{equation}{section}
\newtheorem{theorem}{Theorem} [section]

\newtheorem{proposition}[theorem]{Proposition}

\def\be{\begin{equation}}
\def\ee{\end{equation}}
\def\vp{\varphi}
\def\ve{\varepsilon}
\def\nint{\mathop{\diagup\kern-13.0pt\int}}
\def\textnoint{\matpop
{\raise1.25pt\hbox{${\ssize\diagup}$}\kern-9.5pt\int}}

\allowdisplaybreaks
\newtheorem{lemma}[theorem]{Lemma}

\begin{document}
\parskip=\medskipamount
\title{ On a paper of Erd\"os and Szekeres}
\author{J. Bourgain}
\address{Institute for Advanced Study, Princeton, NJ 08540}
\email{bourgain@math.ias.edu}
\author{M.-C. Chang}
\address{Department of Mathematics, University of California, Riverside, CA 92507}
\email{mcc@math.ucr.edu}
\begin{abstract}
Propositions 1.1 --  1.3 stated below contribute to results and certain problems considered in \cite {E-S}, on the behavior of products
$\prod^n_1 (1-z^{a_j}), 1\leq a_1\leq \cdots\leq a_n$ integers.
In the discussion below,  $\{a_1, \ldots, a_n \}$ will be either a proportional subset of $\{1, \ldots, n\}$ or a set of
large arithmetic diameter.

\end{abstract}
\maketitle

\section
{\bf Introduction}

The aim of this paper is to revisit some of the questions put forward in the paper \cite {E-S} of Erdos and Szekeres.

Following \cite {E-S}, define
$$
M(a_1, \ldots, a_n) =\max_{|z|=1} \prod^n_{i=1} |1-z^{a_i}|
\eqno{(1.1)}
$$
where we assume $a_1\leq a_2\leq\cdots\leq a_n$ positive integers (in this paper, we restrict ourselves to distinct integers $a_1<\cdots < a_n$).

Denote
$$
f(n)=\min_{a_1\leq \cdots \leq a_n} M(a_1, \ldots, a_n)\;\;\; \text { and } \;\;\;\ f_*(n) =\min_{a_1<\cdots< a_n} M(a_1, \ldots, a_n).\eqno{(1.2)}
$$
It was proven in \cite{E-S} that
$$
f(n) \geq \sqrt{2n}.\eqno{(1.3)}
$$
This lower bound remains presently still unimproved.

In the other direction, \cite {E-S} establish an upper bound
$$
f(n)<\exp (n^{1-c}) \text { for some } \ c>0.\eqno{(1.4)}
$$
Subsequent improvements were given by Atkinson \cite {A}
$$
f(n)=\exp \{O(n^{\frac 12}\log n)\}\eqno{(1.5)}
$$
and Odlyzko \cite{O}
$$
f(n) =\exp \{O(n^{\frac 13} (\log n)^{4/3})\}.\eqno{(1.6)}
$$
Also to be mentioned is a construction due to Kolountzakis (\cite {Kol2}, \cite{Kol4}) of a sequence $1<a_1<\cdots <a_n< 2n + O(\sqrt n)$ for which
$$
f_*(n)\leq M(a_1, \ldots, a_n)< \exp \{O(n^{\frac 12} \log n)\}\eqno{(1.7)}
$$
(Note that Odlyzko's construction does not come with distinct frequencies).

As shown by Atkinson \cite{A}, there is a relation between the \cite{E-S} problem and the {\it cosine-minimum problem}.

Define
$$
M_2(n) =\inf \{-\min_\theta \sum^n_{j=1} \cos a_j\theta\}\eqno{(1.8)}
$$
with infinum taken over integer sets $a_1<\cdots <a_n$.

Then
$$
\log f_*(n) < O(M_2(n)\log n).\eqno{(1.9)}
$$
The problem of determining $M_2(n)$ was put forward by Ankeny and Chowla \cite {C1} motivated by questions on zeta
functions.

It is known that $M_2(n)= O (n^{\frac 12})$ and conjectured by Chowla that in fact $M_2(n)\sim n^{\frac 12}$ \cite{C2}.
The current best lower bound is due to Ruzsa \cite{R}
$$
M_2(n)>\exp (c\sqrt{\log n})\eqno{(1.10)}
$$
for some $c>0$.

As pointed out in \cite{O}, polynomials of the form (1.1) are also of interest in connection to Schinzel's problem
\cite{S} of bounding the number of irreducible factors of a polynomial on the unit circle in
terms of its degree and $L^2$-norm.

Propositions 1.1 and 1.2 in this paper establish new results for `dense' sets $S=\{a_1<\dots <a_n\}$.
The former improves upon (1.7).

\begin
{proposition}
There is a subset $\{a_1<\cdots< a_n\} \subset \{1, \ldots, N\}, n\asymp \frac N2$, such that
$$
M(a_1, \ldots, a_n)< \exp(c\sqrt n\sqrt{\log n} \log\log n).\eqno{(1.11)}
$$
\end{proposition}

On the other hand, the following holds

\begin{proposition}
There is a constant $\tau>0$ such that if $\{a_1< \ldots < a_n\}\subset\{1, \ldots, N\}$ and
$n> (1-\tau)N$, then
$$
M(a_1, \ldots , a_n)>\exp \tau n.\eqno{(1.12)}
$$
\end{proposition}

The latter result generalizes the comment made in \cite {E-S} that
$$
\lim_{n\to \infty} [M(1, 2, \ldots, n)]^{1/n}
\eqno{(1.13)}
$$
exists and is between 1 and 2.

In converse direction, one may prove new lower bounds on $M(a_1, \ldots, a_n)$ assuming that the set $\{a_1< \cdots < a_n\}$
has a sufficiently large arithmetic diameter.

First, we are recalling the notion of  a `{\it dissociated set}'
of integers.
We say that $D=\{\nu_1, \ldots, \nu_m\}\subset\mathbb Z$ is dissociated provided $D$ does not admit non-trivial
$0, 1, -1$ relations.
Thus
$$
\ve_1\nu_1+\cdots+\ve_m \nu_m =0 \ \text { with } \ \ve_1=0, 1, -1\eqno{(1.14)}
$$
implies
$$
\ve_1 =\cdots = \ve_m=0.
$$

A more detailed discussion of this notion and its relation to lacunarity appears in \S5 of the paper.

\begin{proposition}
Assume $\{a_1 <\cdots < a_n\}$ contains a dissociated set of size $m$. Then
$$
\log M(a_1, \ldots, a_n)\gg \frac {m^{\frac 12-\ve}}{(\log n)^{1/2}}.\eqno{(1.15)}
$$
\end{proposition}

Hence (1.15) improves upon (1.3) as soon as
$$
m\gg (\log n)^{3+\ve}.\eqno{(1.16)}
$$

\section
{\bf Preliminary estimates}

Let$$
z=e(\theta)= e^{2\pi i\theta}.
$$
By taking the real part of $\rm{Log}(1-e^{2\pi i\theta})=-\sum_{k=1}^{\infty}\frac 1ke^{2\pi ik\theta}$, we have$$
\log |1-z|=  -\sum^\infty_{k=1} \frac {\cos 2\pi k\theta}k.
$$
Therefore, we have

\noindent
{\bf Fact 1.}$$
\prod^n_{j=1} |1-z^{a_j}| = e^{-\sum^n_{j=1} \sum^\infty_{k=1} \ \frac {\cos 2\pi ka_j\theta}k}.\qquad\qquad\qquad\qquad\qquad\qquad
$$

We first establish some preliminary inequalities for later use.

Since the function $e^x$ is convex, we obtain for any probability measure $\mu$ on $\mathbb T$  that
$$
\prod^n_{j=1} |1- e(a_j\theta)|* \mu \geq e^{-(\sum^n_{j=1} \sum^\infty_{k=1} \ \frac {\cos 2\pi k a_j\;\cdot\;}k)*\mu(\theta)}
$$
and therefore we have

\noindent
{\bf Fact 2.}
$$
\Big\Vert \prod^n_{j=1} |1-e(a_j\theta)|\Big\Vert_\infty \geq e^{-\underset{\theta}\min \{\sum^n_{j=1}\sum^\infty_{k=1} \frac
{\cos 2\pi k a_j\;\cdot\;}k *\mu\}(\theta)}.\qquad\qquad\qquad\qquad\qquad
$$

\begin{lemma}
$$
\log |1-e^{2\pi i\theta}|\leq -\sum^J_{j=1} \frac {\rho^j} j \cos 2\pi j\theta +O\Big(\frac
1{\sqrt J}\Big)\eqno{(2.1)}
$$
where $\rho =1 -\frac 1{\sqrt J}$ and (2.1) is valid for all $\theta$.
\end{lemma}

\begin{proof}
We rely on a calculation that appears in \cite {O}, Proposition 1.

Use the inequality \big(\cite {O}, (2.4)\big)
$$
\Big|\frac {1-e^{i\theta}}{1-\rho e^{i\theta}}\Big|\leq \frac 2{1+\rho}  \ \text { for } \
\theta \in [0, 2\pi], 0<\rho<1.\eqno{(2.2)}
$$
From (2.2)
$$
\begin{aligned}
\log |1-e^{i\theta}| &\leq \log |1-\rho e^{i\theta}|+\log \frac 2{1+\rho}\\
&= -\sum^\infty_{j=1} \frac {\rho^j}j \cos j\theta +\log \frac 2{1+\rho}\\
&\leq -\sum^J_{j=1} \frac {\rho^j}j \cos j\theta +\frac {\rho^J}{J(1-\rho)}
+C(1-\rho)
\end{aligned}
\eqno{(2.3)}
$$
by partial summation and since
$$
\log \frac 2{1+\rho}= -\log \Big(1-\frac {1-\rho}2\Big).
$$
Thus (2.1) follows from (2.3) with $\rho$ as above.

\end{proof}

\begin{proposition}
There is a subset $\{a_1 \ldots  a_m\}\subset \{1, \ldots, n\}$ of size
$$
m\asymp \frac n2
$$
and
$$
\Big\Vert \prod^m_{k=1} |1-z^{a_k}| \, \Big\Vert_{L^\infty(|z|=1)}
\leq e^{c\sqrt n \,\sqrt  {\log n})(\log\log n)}.\eqno{(2.4)}
$$
\end{proposition}

\noindent
{\bf Remark.}
(2.4) is a slight improvement of the estimate
$$
\Big\Vert\prod^m_{k=1} |1-z^{a_k}| \Big\Vert_{L^\infty (|z|=1)}\leq e^{c\sqrt n\log n}
$$
resulting from a construction in \cite{Kol1}, p. 162 of a set $\{a_1, \ldots, a_m\}$ as above
and such that
$$
\sum^m_{k=1} \cos 2\pi a_k\theta\geq -c\sqrt m
$$
and Lemma 2.1
$$
\begin{aligned}
\log\prod^m_{k=1}|(1-2a_k)|&\leq -\sum^J_{j=1} \frac {\rho^j}j \sum^m_{k=1} \cos 2\pi
a_k(j\theta)+O\Big(\frac m{\sqrt J}\Big)\\
&\leq C(log J)\sqrt m+O\Big(\frac m{\sqrt J}\Big)\\
&< C\log n\;\sqrt n,
\end{aligned}
$$taking $J=m^2$.

\medskip
\noindent
{\bf Proof of Proposition 2.2.}
Take independent selectors $(\xi_j)_{1\leq j<n}$ with values $0, 1$ and mean $\mathbb E[\xi_j]
=1-\frac jn$.
Let $F_n(\theta)=2\sum_{0 <j<n} (1-\frac jn) \cos 2\pi j\theta+1$ be the Fejer kernel
$$
\sum^m_{k=1} \cos a_k\theta =\sum^n_{\ell =1} \xi_\ell \cos \ell \theta = \frac 12
F_n(\theta)-\frac 12+ \sum^n_{\ell =1} (\xi_\ell -\mathbb E[\xi_\ell])\cos \ell\theta.
\eqno{(2.5)}
$$

By Lemma 2.1 (applies with $J= n^{10} $)
$$
\begin{aligned}
\sum^m_{k=1} \log |1-e^{2\pi ia_k\theta}| \leq
-\sum^J_{j=1} \sum^m_{k=1} \frac {\rho ^j}j \cos 2\pi j a_k\theta+ O\Big(\frac m{\sqrt
J}\Big)
\end{aligned}
\eqno{(2.6)}
$$
and we take $J$ at least $n$ to bound the last term in the right hand side of (2.5) by $\sqrt n$.
We analyze the first term.
Inserting (2.5) gives the sum of the following two expressions ((2.7) and (2.8))
$$
-\sum^J_{j=1} \frac {\rho^j}j \Big(\frac 12 F_n(j\theta)-\frac 12\Big)
\eqno{(2.7)}
$$
$$
-\sum^J_{j=1} \sum^n_{\ell=1} \frac {\rho^j}j (\xi_\ell -\mathbb E[\xi_\ell])\cos 2\pi
\ell j\theta.\eqno{(2.8)}
$$
Since $F_n(j\theta)\geq 0$, (2.7) $\leq \log J$.

Rewrite
$$
(2.8) = - \sum^n_{\ell=1}( \xi_\ell -\mathbb E[\xi_\ell])\Big[\sum^J_{j=1} \frac {\rho^j}j \cos
2\pi j\ell\theta\Big].\eqno{(2.9)}
$$
Note that all frequencies in (2.9) are bounded by $nJ$.

Applying the probabilistic Salem-Zygmund inequality \cite{Kol3} shows that with large probability
$$
(2.9) \lesssim \sqrt{\log \, n J} \Big[\sum^n_{\ell =1} \Big|\sum^J_{j =1}
\frac {\rho^j}j \cos 2\pi j\ell\theta\Big|^2\Big]^{\frac 12}.\eqno{(2.10)}
$$

Our next task is to evaluate the expression $\sum^n_{\ell=1}  \Big|\sum^J_{j =1}
\frac {\rho^j}j \cos 2\pi j\ell\theta\Big|^2$.

A first observation is that we can assume
$$
\Vert\theta\Vert>\frac 1{10n}\eqno{(2.11)}
$$
since otherwise
$$
|1-e^{2\pi i a_k\theta}|\leq 2\pi a_k\Vert\theta\Vert<\frac {2\pi}{10}<1
$$
for all $k=1, \ldots, m$, and also the left hand side of (2.4) is bounded by $1$.

Next, we note that (since $\rho =1-\frac 1{\sqrt J}$)
$$
\begin{aligned}
\Big|\sum^J_{j=1} \frac {\rho^j} j\cos 2\pi j\ell\theta\Big| &\leq \big|\log| 1-\rho
e(\ell\theta)|\big|+\frac {\rho^J}{J(1-\rho)}\\
&<\big|\log| 1-\rho e(\ell \theta)|\big|+1.
\end{aligned}
$$
Hence
$$
\sum^n_{\ell =1} \Big|\sum^J_{j =1}
\frac {\rho^j}j \cos 2\pi j\ell\theta\Big|^2 \lesssim \sum^n_{\ell=1} \big|\log |1-\rho e(\ell
\theta)|\big|^2+n.\eqno{(2.12)}
$$
Fix $\theta $ and for $1<R\lesssim \log J$ define  the dyadic set
$$
S_R=\{1\leq\ell \leq n: \big|\log |1-\rho e(\ell \theta)|\big|\sim R\}.
$$
Thus for $\ell \in S_R$
$$
\Vert\ell\theta\Vert<|1-\rho e(\ell\theta)|<e^{-cR}=:\ve.
$$
Let $q\in\mathbb N$ be the smallest integer with $\Vert q\theta\Vert <2\ve$.
It follows that $|S_R|\lesssim\frac nq+1$.
Assuming $q>R^3$, one obtains
$$
\sum_{\ell \in S_R} \big|\log|1-\rho e(\ell\theta)|\big|^2 \lesssim \Big(\frac n{R^3}+1\Big)R^2
$$
with collected contribution (summing over dyadic $R$)
$$
\sim n+(\log J)^2.\eqno{(2.13)}
$$

It remains to consider $\theta$'s with the property that for some large $R$ and $q<R^3$,
$$
\Vert q\theta\Vert < e^{-cR}.
$$
Hence either $\theta$ admits a rational approximation
$$
\Big\vert\theta-\frac aq\Big|< \frac{e^{-cR}}q< e^{-cR}, \;\;q<R^3\; \text{ and }\; (a, q)=1\eqno{(2.14)}
$$
or \big(in (2.14) when $a=0$\big), by (2.11)
$$
\frac 1n\lesssim \Vert\theta\Vert< e^{-cR}.\eqno{(2.15)}
$$
Consider first the case (2.15). Then
$$
|S_R|\leq |\{\ell = 1, \ldots, n:\Vert\ell\theta\Vert <e^{-cR}\}|\lesssim n e^{-cR}
$$
and the above estimate still holds.

Assume next that $\theta$ satisfies (2.14). Write
$$
\theta=\frac aq +\psi \ \text { with } \ \beta=|\psi|< e^{-cR}.\eqno{(2.16)}
$$

\noindent First, we consider the case $\beta\gtrsim \frac 1{nq}$.

Let $V\subset \{ 1, \ldots, n\}$ be an interval of size $\sim \frac 1{q\beta}$ so that $\{\ell\theta: \ell \in
V\}$ consists of $q\beta$-separated points filling a fraction of $[0, 1]$ (mod 1). Hence
$$
\begin{aligned}
\sum_{\ell\in V} \big|\log |1-\rho e(\ell\theta)|\big|^2 &\lesssim \frac 1{\beta q} \int^1_0 \big|\log |1-\rho
e(t)|\big|^2 dt+\log ^2 (1-\rho)\\
&\lesssim \frac 1{\beta q} +\log^2J
\end{aligned}
$$
and
$$
\sum^n_{\ell=1} \big|\log| 1-\rho e(\ell\theta)|\big|^2 \lesssim n+ nq\,\beta\log^2 n\lesssim n
$$
unless
$$
q\beta\log^2 n>1, \text { i.e. } \log n> e^{cR} \ \text { or } \ R\lesssim \log\log n
$$
where we used (2.14).
Thus if $\beta \gtrsim \frac 1{nq}, (2.12) \lesssim n(\log\log n)^2$.

\noindent The next case is $\beta<\frac 1{100nq}$.

It follows that for $1\leq \ell \leq n$
$$
\Big|\ell\theta -\frac {\ell a} q\Big| <\frac 1{100 q}.\eqno{(2.17)}
$$
We obtain
$$
\sum_{q\nmid \ell}\big|\log|1-\rho e(\ell\theta) |\big|^2 \lesssim n\int^1_0 \big|\log |1-\rho e(t)|\big|^2
dt\lesssim n
$$
and
$$
\begin{aligned}
\sum_{q|\ell}\big|\log |1-\rho e(\ell\theta)|\big|^2\sim&\; \frac 1{q\beta} \int_0^{n\beta} \big|\log|1-\rho
e(t)|\big|^2 dt\\\leq\:
&\frac 1{q\beta}\int_0^{n\beta} \Big(\log\frac 1t\Big)^2 dt\\\lesssim\;
&\frac nq (\log n\beta)^2.
\end{aligned}
\eqno{(2.18)}
$$

We obtain again a bound $O(n)$ unless
$$
|\log n\beta|>\sqrt q
$$
i.e.
$$
\beta< \frac{ e^{-\sqrt q}}n.\eqno{(2.19)}
$$
Thus (2.17) may be replaced by
$$
\Big|\ell\theta -\ell\frac aq\Big| < e^{-\sqrt q} \ \text { for } \ 1\leq \ell\leq n.\eqno{(2.20)}
$$
For $\theta$ satisfying (2.20) we proceed in a different way. Write
$$
\begin{aligned}
\prod |1-e(a_k\theta)|&= \prod^n_{j=1} |1-e(j\theta)|^{\xi_j}\\
&\lesssim \prod^n_{j=1} \Big(\Big|1-e\Big(j\frac aq\Big)\Big| + \frac 1{q^{10}}\Big)^{\xi_j}.
\end{aligned}
\eqno{(2.21)}
$$
We replace $\xi_j$ by its expectation $\mathbb E[\xi_j] = 1-\frac jn$ using again a random argument.
Thus if
$$
\prod^n_{j=1} \Big(\Big| 1-e\Big(j\frac aq\Big)\Big| +\frac 1{q^{10}}\Big)^{1-\frac jn}\eqno{(2.22)}
$$
we have
$$
|\log (2.21)- \log (2.22) |\leq \bigg|\sum^n_{j=1} \Big(\xi_j-\mathbb E[\xi_j]\Big) \log \Big( \Big|1-e\Big(j\frac
aq\Big)\Big|+\frac 1{q^{10}}\Big)\bigg|.\eqno{(2.23)}
$$
Recall that $q<R^3 \lesssim (\log J)^3\sim (\log n)^3$.
Thus with high probability we may bound (2.23) by $c\sqrt n$ $\sqrt{\log\log n}\log q< c\sqrt n(\log\log n)^3$.

Hence
$$
(2.21)\;\leq \;e^{c\sqrt n(\log\log n)^3}(2.22).
$$
Partition $\{1, \ldots, n\}$ in intervals $I=[rq, (r+1)q-1]$ and estimate for each such interval
$$
\begin{aligned}
&\;\prod_{j\in I} \Big(\Big|1-e\Big(j\frac aq\Big)\Big|+ \frac 1{q^{10}}\Big)^{1-\frac jn}\\
\leq&\;q^{c\frac {q^2}n} \Big[\frac 1{q^{10}} \prod_{s=1}^{q-1} \Big(\Big|1-e\Big(s\frac aq\Big)\Big|+\frac
1{q^{10}}\Big]^{1-\frac {rq}n}\\
\leq&\;q^{c\frac {q^2}{n}} \Big[\frac 1{q^{10}} \prod_{s=1}^{q-1}\Big|1-e\Big(\frac sq\Big)\Big|\Big]^{1-\frac
{rq}n}.
\end{aligned}
\eqno{(2.24)}
$$

The product $\prod_{s=1}^{q-1} \big|1-e\big(\frac sq \big)\big|$ may be evaluated using Lemma 2.1 taking $J=q^2, \rho =1-\frac 1q$.
Thus clearly
$$
\begin{aligned}
\sum^{q-1}_{s=1} \log \Big| 1-e\Big(\frac sq\Big)\Big| &\leq -\sum^J_{j=1} \frac {\rho^j} j
\sum_{s=1}^{q-1} \cos 2\pi j\frac sq +O(1)\\
&\leq\sum_{\substack{1\leq j\leq J\\ q\nmid j}} \frac {\rho^j}j +q \sum_{\substack{1\leq j\leq J\\ q|j}}
\frac {\rho^j}j+O(1)\\
&< \log q+C
\end{aligned}
$$
implying that
$$
(2.24)< q^{c\frac {q^2}n} \Big(\frac 1{q^{10}} e^{\log q+c}\Big) ^{1-\frac {rq}n} < q^{c\frac {q^2}n}.
\eqno{(2.25)}
$$

Since (2.22) is obtained as product of (2.24), (2.25) over the intervals $I$, we showed that
$$
(2.22) < q^{c\frac {q^2}n n} 2^q < e^{(\log n)^3}.
$$
Thus the preceding shows that if $\theta$ satisfies (2.20), then
$$
\prod |1-e(a_k\theta)|< e^{c\sqrt n(\log\log n)^3}.\eqno{(2.26)}
$$
Going back to (2.10), omitting the case (2.20) estimated by (2.26), we obtained the bound $cn(\log\log n)^2$ on
(2.12) which permits to majorize (2.8) by $c\sqrt{n\log n} (\log\log n)$ and \hfill\break
$\prod |1-e(a_k\theta)|$ by
$e^{c\sqrt{n\log n}\log\log n}$.
This completes the proof of Proposition 2.2.

\section
{\bf Almost full proportion}

It was observed in \cite {E-S} that
$$
\lim_{n\to\infty}M (1, \ldots, n)^{\frac 1n}\eqno{(3.1)}
$$
exists and lies strictly between 1 and 2.

This fact is in contrast with Proposition 2.2 which gives a subset $S\subset \{1, \ldots, n\}, |S|\asymp\frac n2$ s.t.
$$
\log M(S) \lesssim \sqrt n(\log n)^{\frac 12}\log\log n.\eqno{(3.2)}
$$
However

\begin{proposition}
There is a constant $\tau>0$ such that if $S\subset \{1, \ldots, n\}$ satisfies $|S|>(1-\tau)n$, then
$$
\log M(S)>cn\eqno{(3.3)}
$$
for some $c>0$.
\end{proposition}

Thus (3.3) generalizes (3.1) in some sense, but in view of (3.2), it fails dramatically  if we do not assume $1-\frac {|S|}n$ small enough.

\noindent
{\bf Proof of Proposition 3.1.}

It will be convenient to use Fact 2  for an appropriate $\mu$-convolution, which allow us to estimate the tail
 contribution in the $k$-summation.

Thus consider
$$
\begin{aligned}
&-\min_\theta \Big\{\sum_{j\in S} \, \sum_{k=1}^\infty \frac {\cos 2\pi kj\cdot}k *\mu\Big\}(\theta)\\[5pt]
= &-\min_\theta\sum_{k=1}^{\infty} \, \sum_{j\in S} \, \frac {\hat\mu(jk)}k \cos 2\pi kj\theta\\[5pt]
\end{aligned}$$
$${}\qquad\qquad\qquad\qquad\geq  -\min_\theta \sum^{k_0}_{k=1} \, \sum^n_{j=1} \, \frac {\hat \mu (jk)}k \cos 2\pi kj\theta\qquad \qquad\qquad\;\;\;\quad\eqno{(3.4)}$$
$${}\quad-(\log k_0)\pi n\qquad \qquad\qquad\qquad\qquad
$$
$$
{}\qquad\quad-\sum_{k>k_0} \, \sum^n_{j=1} \, \frac {|\hat \mu(jk)|}k \qquad \qquad\qquad\qquad\qquad \eqno{(3.5)}
$$
since we assumed $|S|> (1-\tau)n$.

Separating in (3.4) the cases $k=1,$ and $ 2\leq k\leq k_0$, we write
$$
\begin{aligned}
(3.4)\geq &-\Big(\sum^n_{j=1} \cos 2\pi j\theta\Big) -\sum^n_{j=1} |1-\hat\mu (j)|\\
&-\sum^{k_0}_{k=2} \, \frac 1k\Big|\sum^n_{j=1} \, \hat \mu (jk) \cos 2\pi kj\theta\Big|.
\end{aligned}
\eqno{(3.6)}
$$
Take $\mu= F_{nR}(\theta)$, $R>1$ an appropriate constant and $F_{nR}(\theta)$ the F\'ejer kernel.

Thus
$$
\begin{aligned}
\widehat F_{nR} (s) &=1-\frac {|s|}{nR} \;\; \text { for } \ |s|\leq nR\\
&=0 \;\;\;\;\;\;\;\;\;\;\;\;\text { otherwise}.
\end{aligned}
$$
Take $\theta =\frac 3{4n}$.
The first term in (3.6) becomes, since
$$
\sum^n_{j= 1} \cos jx =\frac 12 D_n(x) -\frac 12, \;\;\;\text{where } \;D_n(x)=\frac {\sin (n+\frac 12)x}{\sin\frac x2}
$$
is the Dirichlet kernel,
$$
\frac 12-\frac 12 \frac {\sin \frac {3\pi}{2n}(n+\frac 12)}{\sin \frac {3\pi}{4n}}
\sim + \frac 1{2\sin \frac {3\pi}{4n}}.
$$
The second term is
$$
-\sum^n_{j=1} \, \frac j{nR} =- \frac {n+1}{2R}.
$$
The third term becomes
$$
-\sum^{k_0}_{k=2} \frac 1k\Big|\sum^n_{j=1} \Big(1-\frac {jk}{nR}\Big)_+ \cos \pi \frac {3kj}{2n}\Big|.\eqno{(3.7)}
$$
By partial summation, the inner sum is bounded by
$$
\begin{aligned}
&\max_{j_1\leq \min (n, \frac {nR}k)} \Big|\sum^{j_1}_{j=1} \cos \pi \frac {3kj}{2n}\Big|\\[5pt]
=&\max_{j_1 \leq \min (n, \frac {nR}k)} \Big|\frac 12 D_{j_1} \Big(\frac 32 \pi \frac kn\Big) -\frac 12\Big|\\[5pt]
\leq &\;\frac 1{2|\sin \frac 34 \pi \frac kn|} +\frac 12.
\end{aligned}
$$
For $k<k_0 =o(n)$, the first term
$$
\sim\frac 1{2k\sin\frac {3\pi}{4n}}.
$$
Hence
$$
\begin{aligned}
(3.7) &\geq -\sum^{k_0}_{k=2} \, \frac 1{2k^2} \, \frac 1{\sin \frac {3\pi}{4n}}-\log k_0\\
&\geq -\frac 1{2\sin \frac {3\pi}{4n}} \Big(\frac {\pi^2}6 -1\Big) -\log k_0.
\end{aligned}
$$
It follows from the preceding that
$$
\begin{aligned}
(3.4) &\geq +\frac 1{2\sin \frac {3\pi}{4n}} \Big(2-\frac {\pi^2}6\Big) -\log k_0-\frac {n+1}{2R}\\
& =cn-\log k_0
\end{aligned}
$$
for $R$ a sufficiently large constant.

We bound (3.5) by
$$
(3.5) \geq -\sum_{k\geq k_0} \, \frac 1k \, \sum_{j\leq \frac {nR}k} 1\geq -\sum_{k\geq k_0} \, \frac {nR}{k^2} \geq-\frac R{k_0}n.
$$

In summary, we proved that
$$
-\sum_{j\in S} \, \sum^\infty_{k=1} \, \frac {\hat\mu (jk)}{k} \cos 2\pi jk\, \frac 3{4n}
\geq cn -\log k_0 -\tau (\log k_0)n-\frac {C'n}{k_0}
>\frac c2 n
$$
be choosing first $k_0$ large enough and then assuming $\tau$ sufficiently small.

This proves Proposition 3.1.

\section
{\bf Sets with large arithmetical Diameter}

As we pointed out the general lower bound $M(a_1, \ldots, a_n)>\sqrt n$ remains unimproved.
However Proposition 4.1 stated below shows that in certain cases one can do better.

First, we give the following definition.

\noindent{\bf Definition.}
{\sl $D=\{v_1, \ldots, v_m\}\subset\mathbb Z$ is called dissociated provided the relation
$$
\ve_1 v_1+\cdots +\ve_mv_m =0\;\;\;\ \text { with } \ \ve_i=0, 1, -1
$$
implies that $\ve_1=\cdots=\ve_m=0$.}

\medskip

We note that Hadamard lacunary sets are dissociated.

\begin{proposition}
Assume $S=\{a, \ldots, a_n\} $ contains a dissociated set $D$ of size $m$.
Then
$$
\log M(a_1, \ldots, a_n)\gg \frac {m^{\frac 12 -o(1)}}{(\log n)^{\frac 12}}.\eqno{(4.1)}
$$
\end{proposition}

Thus (4.1) improves the general lower bound from \cite {E-S} provided $m>(\log n)^{3+\ve}$.

\noindent
{\bf Remark.}
By a result of Pisier \cite{P}, our assumption is equivalent to $S$ containing a Sidon set $\Lambda$ of size $|\Lambda|\sim m$.
Here `Sidon set' is in the harmonic analysis sense i.e.
$$
\Big\Vert\sum_{n\in\Lambda} \lambda_n e(n\theta)\Big\Vert_\infty \geq c\sum|\lambda_n| \text { for all scalars $\{\lambda _n$\}}
$$
with $c=c(\Lambda)$ to be considered as a constant. (This concept is  different from the Sidon sets in combinatorics!).

Dissociated sets are Sidon and conversely, Pisier proved that if $\Lambda$ is a finite Sidon set, then $\Lambda$ contains
a proportional dissociated set.

\noindent{\bf Proof of Proposition 4.1.}

We derive (4.1) from the equivalent statement
$$
\max_\theta \big(\log |1-e(a_1 \theta)|+ \cdots+ \log |1-e(a_n\theta)|\big) \gg \frac {m^{\frac 12-o(1)}}{(\log n)^{1/2}}\eqno{(4.2)}
$$
which, since $\int\log |1-e(a\theta)|=0$ for $a\in\mathbb Z\backslash \{0\}$, is a consequence of the stronger claim that
$$
\Vert F\Vert_1 \gg \frac {m^{\frac 12-o(1)}}{(\log n)^{1/2}}\eqno{(4.3)}
$$
denoting
$$
F(\theta)=\log |1-e(a_1 \theta)|+\cdots+\log |1-e (a_n \theta)|.
$$
Recall that by Fact 1
$$
F(\theta)= -\sum^\infty_{k=1} \frac 1k f(k\theta)\eqno{(4.4)}
$$
with
$$
f(\theta)=\sum^n_{j=1} \cos (2\pi a_j\theta).
$$

We first perform a finite Mobius inversion on (4.4).
Recall that
$$
\sum_{\substack{d|k, d\leq r\\ d\text { square free}}} \mu (d) =\begin{cases}
1\quad\text { if } k=1\\ 0\quad\text { if } 1<k\leq r\end{cases}
$$
Hence
$$
\begin{aligned}
\sum_{\substack{ d< r\\ \text { square free}}} F(d\theta) \frac {\mu (d)}d& = -\sum^n_{j=1} \sum^\infty_{k=1} \sum_{\substack{d< r\\ \text { square free}}} \cos (2\pi a_j dk\theta) \frac {\mu(d)}{dk}
\\[5pt]&=-\sum^n_{j=1} \sum^\infty_{\ell =1} \frac {\cos (2\pi  a_j\ell\theta)} \ell \bigg[\sum_{\substack{d|\ell, d<r\\ \text { square free}}}\mu(d)\bigg]\\[5pt]
& = - f(\theta) -\sum^n_{j=1} \sum_{\ell>r} \frac {\cos (2\pi a_j\ell\theta)}\ell \bigg[\sum_{\substack{d|\ell, d< r\\ \text { square free}}} \mu(d)\bigg]
\\[5pt]&=-f(\theta)+G(\theta),
\end{aligned}\eqno{(4.5)}
$$
where$$G(\theta)=-\sum^n_{j=1} \sum_{\ell>r} \frac {\cos (2\pi a_j\ell\theta)}\ell \bigg[\sum_{\substack{d|\ell, d< r\\ \text { square free}}} \mu(d)\bigg].$$
Note also that
$$
\Big| \sum_{\substack{d|\ell, d<r\\ \text { square free}}} \mu(d)\Big|\leq  2^{\omega(\ell)},\eqno{(4.6)}
$$
where $\omega(\ell)$ is the number of distinct prime factors of $\ell$.

Denote $m$ the size of the largest dissociated set contained in $\{a_1, \ldots, a_n\}$.
Our first task will be to bound the Fourier transform $\Vert\hat G\Vert_\infty$ of $G$.

Thus given $t\in\mathbb Z$, we have
$$
|\hat G(t)|\leq \frac 12\sum^n_{j=1} \frac {a_j}t 2^{\omega(\frac t{a_j})}.\eqno{(4.7)}
$$
We will bound (4.7) by considering dyadic ranges, letting for $K>r$ dyadic
$$
J=J_K=\{j\in[1, n]\;: \;a_j|t \;\;\text{ and } \;\;\frac t{a_j}\sim K\}.
$$
Thus
$$
\begin{aligned}
\sum_{j\in J} \frac {a_j}t 2^{\omega (\frac t{a_j})} & \leq \sqrt{\sum_{j\in J}\Big(\frac {a_j}t\Big)^2} \,
\Big(\sum_{k\leq K} 4^{\omega (k)}\Big)^{\frac 12}\\
&\lesssim |J|^{\frac 12} K^{-1} K^{\frac 12} (\log K)^{2} =\bigg(\frac {|J|}K\bigg)^{\frac 12} (\log K)^2.
\end{aligned}
\eqno{(4.8)}
$$
Assume
$$
|J|>\frac K{(\log K)^{8}}.\eqno{(4.9)}
$$
Our aim is to get a contradiction for appropriate choice of $r$.

At this point, we invoke the following result from \cite {H-T} (see $Fq$ {(1.14)}).

Denote
$$\psi(x, y)=\big|\{n\leq x :  \text{ if } \;p|n, \text{ then } p\leq y\}\big|.$$

\begin{lemma}
For any $0<\alpha<1$, we have
$$
\psi\big(x, (\log x)^{1/\alpha}\big) < x^{1-\alpha +o(1)} \ \text { for } \ x\to\infty.\eqno{(4.10)}
$$
\end{lemma}

It follows from (4.9) that for any fixed $1>\alpha>0$, we have
$$
|J|>2\psi \big(K, (\log K)^{\frac 1\alpha}\big).\eqno{(4.11)}
$$
We make the following construction.

By (4.11), there is $j_1\in J$ such that $\frac t{a_{j_1}}$ has a prime divisor $p_1 >(\log K)^{\frac 1\alpha}$ and
we write $\frac t{a_{j_1}} = p_1 b_1$.

Next, let $J_1 =\{j\in J : p_1|\frac t{a_j}\;\}$.
Hence $|J_1|< \frac K{p_1} +1< \frac K{(\log K)^{\frac 1\alpha}} <\frac {|J|}{(\log K)^{\frac 1\alpha - 8}}$ where
we assume $\alpha$ taken much smaller than $\frac 1{8}$.

It follows that also
$$
|J\backslash J_1|>\Big(2-\frac 1{(\log K)^{\frac 1\alpha -8}}
\Big) \psi\big(K, (\log K)^{\frac 1\alpha}\big)
$$
which permits to introduce $j_2 \in J\backslash J_1$ and a prime $p_2>(\log K)^{\frac 1\alpha}$ such that $p_2|\frac t{a_{j_2}}$.
Write $\frac t{a_{j_2}} =p_2b_2$.
Clearly $p_2\not= p_1$ and $p_1\nmid b_2$.

The contribution of the process is clear.
We may introduce elements
$$
j_1, \ldots, j_s\in J \;\text{ with }\; s\gtrsim (\log K)^{\frac 1\alpha-8}
$$
and prime divisors $p_{s'}|\frac t{a_{j_{s'}}}.$ Write $ \frac t{a_{j_{s'}}}= p_{s'} b_{s'}$ such that $p_{s'} \nmid
\frac t{a_{j_{s''}}}$ for $s'<s''$. Hence $p_{s''}\not = p_{s'}$ for $s'\not= s''$ and
$$
p_{s'} \nmid b_{s''} \ \text { for } \ s'< s''.\eqno{(4.12)}
$$
We claim that the set $\{a_{j_1}, \ldots, a_{j_s}\}$ is dissociated.
Otherwise, there is a non-trivial relation
$$
\ve_1 a_{j_1}+\cdots+ \ve_{s} a_{j_s}=0 \qquad\text{with }  \ve_{s'} =0, 1, -1
$$
which by the preceding translates in
$$
\ve_1\frac 1{p_1b_1}+\cdots + \ve_s \frac 1{p_sb_s}=0
$$
or
$$
\sum^s_{s'=1} \ve_{s'} \prod_{s'' \not= s'} p_{s''} b_{s''} = 0.
$$
Let $s_1$ be the smallest $s'$ with $\ve_{s'}\not=0$. Then
$$
\sum^s_{s'=s_1} \ve_{s'} \prod_{\substack {s''\not=s'\\ s''\geq s_1}} p_{s''} b_{s''}= 0.\eqno{(4.13)}
$$
Since $$p_{s_1}\Big|\prod_{\substack{ s''\not= s'\\ s''\geq s_1}} p_{s''} b_{s''}\;\text{ for }\; s'>s_1,$$ identity (4.13) implies $$\;p_{s_1}\Big|\prod_{s''>s_1}
b_{s''},$$ contradicting (4.12).

Hence $\{a_{j_1}, \ldots, a_{j_s}\}$ is dissociated and by definition of $m$,
$$
s\leq m
$$
implying
$$
m\geq (\log K)^{\frac 1\alpha -8} \;\;\text{ and } \; \;\log r\le m^{\frac \alpha{1-8\alpha}}.
$$
Thus, by taking
$$
\log r\sim m^{2\alpha} \quad (\alpha \text { small enough})
$$we obtain a contradiction under assumption (4.9).

Hence
$$
|J_K|< \frac K{(\log K)^{8}} \;\;\ \text { for $K>r$}
$$
and summing (4.8) over dyadic ranges of $K>r$ gives the bound
$$
|\hat G(t)|<\sum_{\substack{K>r\\ \text {dyadic}}} \frac 1{(\log K)^2} \lesssim \frac 1{\log r}.\eqno{(4.14)}
$$
Consequently
$$
\widehat{(4.5)} (t) =- \hat f(t) +O\Big(\frac 1{\log r}\Big) =-\hat f(t) +o(1) \text { for all $t\in \mathbb Z$}.\eqno{(4.15)}
$$
Since
$$
\hat f(j) =\frac 12 ,
$$we have
$$
\widehat{(4.5)} (j) =-\frac 12 +o(1) .\eqno{(4.16)}$$
Next, let $D$ be a size $m$ dissociated set in $\{a_1, \ldots, a_n\}$.
Define
$$
\vp (\theta) =\frac 1{\sqrt m} \sum_{j\in D} \, e(j\theta).
$$
Also, let  $\Phi, \Psi$ be the dual Orliez  functions
$$
\Phi(x)= x\sqrt{\log (2+x)}\;\;\ \text {  and \, \;\;$\Psi(x) =e^{x^2}$}.
$$
It is well known (e.g. Theorem 3.1 in \cite{Rud}.) that
$$
\Vert\vp\Vert_{L^{\Psi}} =O(1).
$$
By (4.16)
$$
\Big(\frac 12- o(1)\Big) \sqrt m \leq \Big|\int_0^1 (4.5) \vp(\theta)d\theta\Big|\leq C\Vert(4.5)\Vert_{L^{\Phi}}\eqno{(4.17)}
$$
It remains to bound $\Vert(4.5)\Vert_{L^{\Phi}}$.

Estimate
$$
\begin{aligned}
&\int |(4.5)|\sqrt{\log (|(4.5)|+2)}\;d\theta\\ \leq\;\;
& \sum_{j>0} 2^{j/2} \, \int_{2^{2^{^{j-1}}}\leq \lambda \leq 2^{2^{^j}}} \mu(M) \;d\lambda,
\end{aligned}
\eqno{(4.18)}
$$
Where $M=\{\theta: (4.5)(\theta)>\lambda\}$ and $\mu$ is the measure.
Using the left hand side of (4.5), the $j$-summands is bounded by
$$
2^{j/2}\Vert(4.5)\Vert_1 \lesssim 2^{j/2} \log r \;\Vert F\Vert_1.\eqno{(4.19)}
$$
Also, let $\Psi_1(u)=e^u$. Then
$$
\begin{aligned}
\bigg\Vert\sum_{d\leq r} \ \frac {|F(d\theta)|}d\bigg\Vert_{L^{\Psi_1}}\leq
 (\log r) \Vert F\Vert_{L^{\Psi_1}}\lesssim n\log r,
\end{aligned}
$$
since $\Vert\log |1-e^{i\theta}|\, \Vert_{L^{\Psi_1}} <\infty$.

Thus also the bound
$$
\mu(M)\leq e^{-c \ \frac \lambda{n\log r}}
$$
implying the following bound for the $j$-summands
$$
2^{j/2} 2^{2^j} e^{-c \, \frac {2^{2^{j-1}}}{n\log r}}.\eqno{(4.20)}
$$
Hence
$$
(4.18) < \sum_j 2^{j/2} \min \Big((\log r)\Vert F\Vert_1, 2^{2^j} e^{-c\, \frac {2^{2^{j-1}}}{n\log r}}\Big).
$$

For $2^{2^{j-2}}<n\log r$, we get the contribution
$$
(\log n)^{\frac 12} \log r\Vert F\Vert_1.
$$
For $2^{2^{j-2}}\geq n\log r$, we bound by
$$
\begin{aligned}
&(n\log r)^{4+\epsilon} e^{-cn\log r} +(n\log r)^{4\cdot 2+\epsilon}\,  e^{-c(n(\log r))^3} +\cdots+ (n\log r)^{4\cdot 2^{u-1}+\epsilon} \, e^{-c(n\log r)^{2^u-1}}+\cdots\\
< \;&O (1).
\end{aligned}
$$
Hence
$$
\Vert(4.5)\Vert_{L^{\Phi}}\lesssim (4.18) < (\log n)^{\frac 12} m^{2\alpha} \Vert F\Vert_1\eqno{(4.21)}
$$
recalling above choice for $\log r$.

Returning to (4.17), we proved that
$$
\Big(\frac 12-o(1)\Big) m^{\frac 12-2\alpha} \lesssim (\log n)^{\frac 12} \Vert F\Vert_1
$$
hence
$$
\Vert F\Vert_1 \gtrsim m^{\frac 12-\ve} (\log n)^{-\frac 12}.
$$
This proves (4.3) and hence Proposition 4.1.

\section*{\bf Acknowledgment}

During the preparation of this paper, the first author was partially supported by the
NSF Grant~DMS~1301619 and the second author by the NSF Grant~DMS~1301608.

\end{document}